\documentclass[12pt,reqno]{article}

\usepackage{amssymb}
\usepackage{amsmath}
\usepackage{amsthm}
\usepackage{amsfonts}
\usepackage{authblk}
\usepackage{hyperref}
\usepackage{color}
\usepackage{float}
\usepackage{graphicx}
\usepackage{enumitem}

\theoremstyle{plain}
\newtheorem{theorem}{Theorem}
\newtheorem{prop}[theorem]{Proposition}
\newtheorem{lem}[theorem]{Lemma}
\newtheorem{cor}[theorem]{Corollary}

\theoremstyle{definition}

\newtheorem{rem}[theorem]{Remark}

\makeatletter
\newcommand{\subjclass}[2][2020]{%
  \let\@oldtitle\@title%
  \gdef\@title{\@oldtitle\footnotetext{#1 \emph{Mathematics subject classification.} #2}}%
}
\newcommand{\keywords}[1]{%
  \let\@@oldtitle\@title%
  \gdef\@title{\@@oldtitle\footnotetext{\emph{Key words and phrases.} #1.}}%
}
\makeatother

\begin{document}

\title{Some remarks related to the density of \\$\{(b^n\pmod n)/n:n\in\mathbb{N}\}$}
\author{Martin Lind}
\affil{Department of Mathematics and Computer Science, Karlstad University, Universitetsgatan 2, 65188 Karlstad, Sweden\\ e-mail: \href{mailto:martin.lind@kau.se}{martin.lind@kau.se}}
\date{}

\maketitle

\begin{abstract}
    For $b\in\mathbb{N}, b\ge2$ we determine the limit points of certain subsets of
    $$
        \left\{\frac{b^n\pmod{n}}{n}:n\in\mathbb{N}\right\}.
    $$
    As a consequence, we obtain the density of the latter set in $[0,1]$, a result first established in 2013 by 
    Cilleruelo, Kumchev, Luca, Ru\'{e} and Shparlinski.
\end{abstract}

\section{Introduction}

Let $b\in\mathbb{N}, b\ge2$. In 2013, Cilleruelo et al. \cite{CKLRS} proved that
\begin{equation}
    \label{Seqb}
    \left\{\frac{b^n\pmod{n}}{n}:n\in\mathbb{N}\right\}
\end{equation}
is dense in $[0,1]$. In fact, they proved a stronger statement. Let $\mathbb{P}=\{2,3,\ldots\}$ denote the prime numbers and set 
\begin{equation}
    \label{Seqb2}
    \left\{\frac{b^{qp}\pmod{qp}}{qp}:p,q\in\mathbb{P}, q\le\log(p)/\log(b)\right\}.
\end{equation}
The main result of \cite{CKLRS} is an explicit rate of decay for the discrepancy of (\ref{Seqb2}). In particular, (\ref{Seqb2}) is equidistributed modulo 1 (in the sense of Weyl); this implies the density of (\ref{Seqb2}). In this direction, we also mention the works of Dubickas \cite{D1, D2} where, among other results, the density is established for much more general sets than (\ref{Seqb}).

In this note, we consider (\ref{Seqb2}) from a different point of view to \cite{CKLRS}. Namely, for a \emph{fixed} $q\in\mathbb{P}$, we determine all possible limits of $b^{pq}\pmod{pq}/pq$ as $p\rightarrow\infty, p\in\mathbb{P}$. (Theorem \ref{derived} below). This result is sufficient to establish the density of (\ref{Seqb}).
For any set $A\subseteq\mathbb{R}$, let $A'$ be the set of all limit points of $A$. ($A'$ is usually referred to as the \emph{derived set} of $A$.)
\begin{theorem}
\label{derived}
Let $b\in\mathbb{N}, b\ge2$ and let $q\in\mathbb{P}$ be a fixed but arbitrary prime number with $q>b^2$. Denote 
\begin{equation}
    \nonumber
    \mathcal{S}_{b,q}=\left\{\frac{b^{qp}\pmod{qp}}{qp}:p\in\mathbb{P}\right\}.
\end{equation}
Then 
\begin{equation}
    \label{derivedSet}
    \mathcal{S}_{b,q}'=\left\{\frac{k}{q}: k=0,1,2,\ldots, q-1\right\}.
\end{equation}
\end{theorem}
Figure \ref{figure} below illustrate the rather striking structure of $\mathcal{S}_{2,13}$. Note in particular the 13 "bands" corresponding to the 13 limit points.
\begin{figure}[H]
\label{figure}
\begin{center}
    \includegraphics[scale=1]{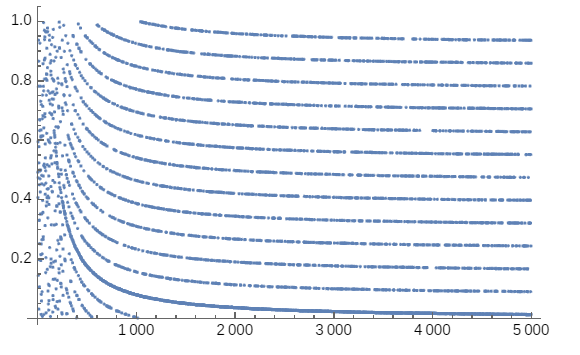}
\end{center}
\caption{$2^{13p}\pmod{13p}/13p$ for the first 5000 prime numbers $p$.}
\end{figure}
We have the following immediate consequence.
\begin{cor}
    The set $\mathcal{S}_b$ defined by (\ref{Seqb}) is dense in $[0,1]$.
\end{cor}
\begin{proof}
    We must show that for any interval $J\subset[0,1]$ there is $s\in\mathcal{S}_b$ such that $s\in J$. Take a prime $q$ such that $1/q<\text{length}(J)$. Then there is $k\in\{0,1,\ldots,q-1\}$ such that $k/q\in J$. By Theorem \ref{derived}, there is a prime $p$ such that $(b^{qp}\pmod{qp})/qp\in J$ (actually, there's an infinite number of such primes).  Since $(b^{qp}\pmod{qp})/qp\in\mathcal{S}_b$, we may take $s=(b^{qp}\pmod{qp})/qp$.
\end{proof}
To calculate the order of the discrepancy of the sequence (\ref{Seqb2}), the authors of \cite{CKLRS} use careful estimates of exponential sums. Our proof of Theorem \ref{derived} is on the other hand completely elementary, except for invoking Dirichlet's theorem on the infinitude of primes in arithmetic progression at one point. 
We discuss in more detail our proof. Denote by $\mathbb{Z}[n]=\mathbb{Z}/n\mathbb{Z}$ and $\mathbb{Z}^*[n]=(\mathbb{Z}/n\mathbb{Z})^*$ the additive group and multiplicative group of integers modulo $n$, respectively.
A crucial step of the proof of Theorem \ref{derived} consists of showing that certain modular equations always have solutions. More specifically, fix a prime number $q>b^2$. For any $k\in\mathbb{Z}[q]$ we consider the equation
\begin{equation}
    \label{solutionEq}
    b^x\equiv kx+b\pmod{q}.
\end{equation}
For each $k\in\mathbb{Z}[q]$ denote by $Z_k$ the set of solutions in 
$\mathbb{Z}^*[q(q-1)]$ 
to the equation (\ref{solutionEq}), i.e.,
\begin{equation}
\label{solSets}
Z_k=\{x\in\mathbb{Z}^*[q(q-1)]: b^x\equiv kx+b\pmod{q}\}.
\end{equation}
In Section 2, we show (Lemma \ref{MainLemma}) that $Z_k\neq\emptyset$ for each $k\in\mathbb{Z}[q]$. Using this fact together with the above mentioned theorem of Dirichlet on primes in arithmetic progressions, Theorem \ref{derived} can be proved. 

In Section 3, we elaborate on Lemma \ref{MainLemma} and determine the exact structure of the sets (\ref{solSets}) (see Theorem \ref{structureThm} below). Since this further analysis is not needed for the proof of Theorem \ref{derived}, we place it in a separate section. Furthermore, Theorem \ref{structureThm} is rather complicated to formulate so we omit its statement here. However, we point out a few of its consequences below.

For Theorem \ref{derived} we require $q>b^2$. In fact, (\ref{derivedSet}) may fail if there are no restrictions on the relationship between $b$ and $q$. For instance, if $b=6$ and $q=7$, then a simulation shows
$$
\lim_{p\rightarrow\infty, p\in\mathbb{P}}\frac{6^{7p}\pmod{7p}}{7p}=0,
$$
i.e. $\mathcal{S}'_{6,7}=\{0\}$. As a consequence of Theorem \ref{structureThm}, we have the following dichotomy.
\begin{prop}
    \label{dichotomy}
    Let $b\in\mathbb{N}, b\ge2$ and and $q\in\mathbb{P}$.
    \begin{enumerate}[label=(\roman*)]
        \item \label{firstitem} If $b^2\not\equiv 1\pmod{q}$, then (\ref{derivedSet}) holds.
        \item \label{seconditem} If $b^2\equiv 1\pmod{q}$, then $\mathcal{S}_{b,q}'=\{0\}$.
    \end{enumerate}
\end{prop}
Note that if $q>b^2$, then automatically $b^2\not\equiv 1\pmod{q}$. Therefore, Proposition \ref{dichotomy} \ref{firstitem} agrees with Theorem \ref{derived}. 

Theorem \ref{structureThm} also provides an immediate way to calculate the number of elements in $Z_k~~(k=0,1,\ldots, q-1)$. In contrast, Lemma \ref{MainLemma} only allows us to conclude $Z_k\neq\emptyset$.
Here and below, $|A|$ denotes the cardinality of the finite set $A$. 
\begin{prop} 
\label{countProp}
Let $b\in\mathbb{N}, b\ge2$ and $q\in\mathbb{P}$. Then
\begin{equation}
    \label{card1}
    |Z_0|=(q-1)m_b(q),
\end{equation}
and
\begin{equation}
    \label{card2}
    |Z_k|=\varphi(q-1)-m_b(q)\quad (k=1,2,\ldots, q-1),
\end{equation}
where $\varphi$ denotes Euler's totient function and $m_b$ is a function defined by (\ref{mbDef}) below. 
\end{prop}
\begin{rem}
We mention that if $b^2\equiv1\pmod{q}$, then the right-hand side of (\ref{card2}) is 0 for every $k=1,2,\ldots,q-1$ and the right-hand side of (\ref{card1}) simplifies to $(q-1)\varphi(q-1)=\varphi(q(q-1))=|\mathbb{Z}^*[q(q-1)]|$. This agrees with Proposition \ref{dichotomy} above.
\end{rem}

\section{Proof of Theorem \ref{derived}}

We shall need the following simple lemma.
\begin{lem} Let $N\in\mathbb{N}~~(N\ge3)$.
    \label{SimpleLemma0}
    \begin{enumerate}[label=(\roman*)]
        \item\label{SimpleLemma}
        If $\alpha\in\mathbb{Z}^*[N-1]$ and $\beta\in\mathbb{Z}^*[N]$, and
        \begin{equation}
            \label{repr}
            x=\alpha N+\beta(N-1)\pmod{N(N-1)},
        \end{equation}
        then $x\in\mathbb{Z}^*[N(N-1)]$.
        \item \label{SimpleLemmaConv}
        Conversely, for any $x\in\mathbb{Z}^*[N(N-1)]$ there exist unique $\alpha\in\mathbb{Z}^*[N-1]$ and $\beta\in\mathbb{Z}^*[N]$ such that (\ref{repr}) holds.
    \end{enumerate}
\end{lem}
\begin{rem}
    To avoid confusion, we make the following remark on notation. As usual, $x\equiv y\pmod{n}$ simply means $n|(x-y)$. When we write $x=y\pmod{n}$, we mean that $x\in\mathbb{Z}[n]$ is the residue of $y$ modulo $n$. 
\end{rem}
\begin{proof}[Proof of Lemma \ref{SimpleLemma0}]
    We first prove statement \ref{SimpleLemma}.
    Let $d$ be an arbitrary divisor of $N(N-1)$. Since $\gcd(N,N-1)=1$ either $d|N$ or $d|(N-1)$. Without loss of generality, we may assume that $d$ is a divisor of $N$. Then $d|(\alpha N+nN(N-1))$ for any $a\in\mathbb{Z}^*[N-1]$ and $n\in\mathbb{Z}$. On the other hand, $d$ does not divide $N-1$, and since $\gcd(\beta,N)=1$ it follows that $d$ does not divide $\beta(N-1)$ for any $\beta\in\mathbb{Z}^*[N]$. Hence, $d$ does not divide $\alpha N+\beta(N-1)+nN(N-1)$. To summarize, any divisor of $N(N-1)$ does not divide $x=\alpha N+\beta(N-1)\pmod{N(N-1)}$, this proves statement \ref{SimpleLemma}.

    We proceed to prove statement \ref{SimpleLemmaConv}
    Take $x\in\mathbb{Z}^*[N(N-1)]$ and define $\beta= N-x\pmod{N}$. Since $\gcd(x,N)=1$ it is easy to see that $\gcd(\beta,N)=1$, i.e., $\beta\in\mathbb{Z}^*[N]$. Furthermore, $N|(x+\beta)$ so $\alpha=(x+\beta)/N-\beta\pmod{N-1}$ is well-defined.
    Thus, there exists $n\in\mathbb{N}$ such that $N\alpha=x-\beta(N-1)+nN(N-1)$. Since $\gcd(x,N-1)=1$, it follows that $\gcd(N\alpha, N-1)=1$. Therefore, $\gcd(\alpha,N-1)=1$. Further, by definition $x=\alpha N+\beta(N-1)\pmod{N(N-1)}$, proving the existence of the representation (\ref{repr}).
    To prove uniqueness of $\alpha,\beta$, assume that 
    $$
        x=\alpha' N+\beta'(N-1)=\alpha'' N+\beta''(N-1)\pmod{N(N-1)}
    $$
    for $\alpha',\alpha''\in\mathbb{Z}^*[N-1]$ and $\beta',\beta''\in\mathbb{Z}^*[N]$.
    It follows that there exists $n\in\mathbb{Z}$ such that 
    $$
        (\alpha'-\alpha'')N+(\beta'-\beta'')(N-1)+nN(N-1)=0
    $$
    Rearranging the above equality, we get that $(N-1)|(\alpha'-\alpha'')N$.
    Since $|\alpha'-\alpha''|<N-1$ and $\gcd(N,N-1)=1$, the only possibility is $\alpha'=\alpha''$. A similar argument shows that $\beta'=\beta''$.
\end{proof}

\begin{lem}
\label{MainLemma}
Let $b\in\mathbb{N}, b\ge2$ and $q\in\mathbb{P}$ satisfies $q>b^2$. Let the sets $Z_k$ be defined by (\ref{solSets}), then
\begin{equation}
    \label{Nonempty}
    Z_k\neq\emptyset\quad\text{for each}\quad k\in\mathbb{Z}[q],
\end{equation}
and
\begin{equation}
    \label{Disjoint}
    \bigsqcup_{k=0}^{q-1}Z_k=\mathbb{Z}^*[q(q-1)].
\end{equation}
(Here, $\sqcup$ denotes disjoint union.)
\end{lem}

\begin{proof}
    We prove (\ref{Nonempty}) first.
    Note that if $k=0$, then we may take $x=2q-1\in\mathbb{Z}[q(q-1)]$ (by Lemma \ref{SimpleLemma0} \ref{SimpleLemma}). Further, by Fermat's little theorem $b^x\equiv b\pmod{q}$ while $b-0\times x=b$, i.e., $x\in Z_0$.
    
    Fix now $k\in\mathbb{Z}[q]\setminus\{0\}$. Since $q-2\in\mathbb{Z}^*[q-1]$, it follows from Lemma \ref{SimpleLemma0} \ref{SimpleLemma} that
    \begin{equation}
        \label{solDef0}
        x=q(q-2)+\beta(q-1)\pmod{q(q-1)}
    \end{equation}
    belongs to $\mathbb{Z}^*[q(q-1)]$ for any $\beta\in\mathbb{Z}^*[q]$. We shall determine a specific $\beta\in\mathbb{Z}^*[q]$ such that $x\in Z_k$.
    Note that
     \begin{equation}
         \label{MainRelation}
         b^x\equiv b^{q-2}\pmod{q}\quad\text{and}\quad b+kx\equiv b-k\beta\pmod{q}.
     \end{equation}
    For the first relation of (\ref{MainRelation}), we used Fermat's little theorem to obtain that
    $$
        b^{q(q-2)+\beta(q-1)+nq(q-1)}\equiv b^{q-2}\pmod{q}.
    $$
    (Note that this also holds when $n<0$, since $b^{-1}\equiv b^{q-2}\pmod{q}$.)

    Denote by 
    $$
        S=\{b-k\beta\pmod{q}:\beta\in\mathbb{Z}^*[q]\}.
    $$
    We shall prove that $S=\mathbb{Z}[q]\setminus\{b\}$. Clearly, $S\subseteq\mathbb{Z}[q]$. Furthermore, $|S|=q-1$. Indeed, let $\beta',\beta''\in\mathbb{Z}^*[q]$ and assume $b-k\beta'\equiv b-k\beta''\pmod{q}$. The latter congruence implies $k(\beta'-\beta'')\equiv 0\pmod{q}$ which is only true if $\beta'=\beta''$, since $q$ is prime. By the same token, $b\notin S$, since $b\in S$ would imply $k\beta\equiv 0\pmod{q}$, which is impossible since $\beta\in\mathbb{Z}^*[q]$ and $k\neq0$. Consequently, $S=\mathbb{Z}[q]\setminus\{b\}$. 
    
    Consider now $\langle b\rangle$, the cyclic subgroup of $\mathbb{Z}^*[q]$ generated by $b$. Denote $\delta=|\langle b\rangle|$. Since $S=\mathbb{Z}[q]\setminus\{b\}$, we have
    \begin{equation}
        \label{SetEq}
          S\cap\langle b\rangle=\langle b\rangle\setminus\{b\}=\{b^l\pmod{q}:l=0,2,3\ldots,\delta-1\}.
    \end{equation}
    By Lagrange's theorem, $q-1=N\delta$ for some $N\in\mathbb{N}$. 
    Therefore
    $$
        b^{q-2}=b^{\delta-1+(N-1)\delta}\equiv b^{\delta-1}\pmod{q}.
    $$
    Since $\delta>2$ we have $\delta-1>1$.
    Therefore, by (\ref{SetEq}), there is a unique $\beta\in\mathbb{Z}^*[q]$ defined by 
    \begin{equation}
        \label{betaDef}
        b-k\beta\equiv b^{\delta-1}\pmod{q}
    \end{equation}
    Clearly $x$ defined by (\ref{solDef0}) with $\beta$ given by (\ref{betaDef}) satisfies $x\in Z_k$. This concludes the proof of (\ref{Nonempty}).
    
    We continue to prove (\ref{Disjoint}). Assume that $x\in\mathbb{Z}^*[q(q-1)]$ satisfies both $b^x\equiv kx+b\pmod{q}$ and $b^x\equiv lx+b\pmod{q}$. Then $(k-l)x\equiv 0\pmod{q}$. Since $x\in\mathbb{Z}^*[q(q-1)]$ the residue $x\pmod{q}\neq 0$ and the only possibility is $k-l\equiv 0\pmod{q}$, whence $k=l$.  This shows $Z_k\cap Z_l=\emptyset$ for $k\neq l$. Take now $x\in\mathbb{Z}^*[q(q-1)]$. We shall prove that there is $k\in\mathbb{Z}[q]$ such that $x\in Z_k$. Since $\gcd(x,q)=1$ there are integers $A,B$ such that $Ax+Bq=1$. Therefore, $b^x-b=(b^x-b)A x+(b^x-b)B q$ and we set $k=k(x)=(b^x-b)A\pmod{q}$. Clearly $b^x-b\equiv kx\pmod{q}$, or equivalently, $x\in Z_k$
\end{proof}

We are now in a position to prove Theorem \ref{derived}.
\begin{proof}[Proof of Theorem \ref{derived}]
Recall that $\mathbb{P}$ denotes the set of prime numbers.
Fix $k\in\mathbb{Z}[q]$, we shall show that there exists a sequence $\{p_n\}_{n=1}^\infty\subset\mathbb{P}$ such that
\begin{equation}
    \label{limit}
    \lim_{n\rightarrow\infty}\frac{b^{qp_n}\pmod{qp_n}}{qp_n}=\frac{k}{q}.
\end{equation}
For any $p\in\mathbb{P}$ we consider the quantity
$$
f_{q,k}(p)=b^{qp}-b^q-kp.
$$
By Fermat's little theorem, we have $f_{q,k}(p)\equiv 0\pmod{p}$ for all $p\in\mathbb{P}$.
By Lemma \ref{MainLemma}, there exists $x\in\mathbb{Z}^*[q(q-1)]$ such that
$$
b^x-kx-b\equiv 0\pmod{q}
$$
Assume that $p\in\mathbb{P}$ satisfies $p\equiv x\pmod{q(q-1)}$, i.e., $p=Nq(q-1)+x$ for some $N\in\mathbb{N}$. By a repeated application of Fermat's little theorem, we obtain
\begin{eqnarray}
    \nonumber
    f_{q,k}(p)&=&b^{Nq^2(q-1)+qx}-b^{q}-k(Nq(q-1)+x)\\
    \nonumber
    &=&(b^{q-1})^{q^2N}\times(b^q)^x-b^q-q(kN(q-1))-kx\\
    \nonumber
    &\equiv&b^x-b-kx\pmod{q}
\end{eqnarray}
Hence, for any $p\in\mathbb{P}$ that satisfies $p\equiv x\pmod{q(q-1)}$ we have 
$$
    f_{q,k}(p)\equiv 0\pmod{p}\quad\text{and}\quad f_{q,k}(p)\equiv 0\pmod{q}.
$$
Since $p,q\in\mathbb{P}$, it follows that
$f_{q,k}(p)\equiv 0\pmod{qp}$. In other words, for any $q\in\mathbb{P}$ and $p\in\mathbb{P}$ with $p\equiv x\pmod{q(q-1)}$, we have
$$
b^{qp}=Mqp+b^q+kp
$$
for some $M\in\mathbb{N}$.
By Dirichlet's theorem on primes in arithmetic progressions \cite[Chapter 7]{A}, for any $x\in\mathbb{Z}^*[q(q-1)]$ there are infinitely many $p\in\mathbb{P}$ with $p\equiv x\pmod{q(q-1)}$. Take a sequence $\{p_n\}$ of such primes sufficiently big to have $b^q+kp_n<qp_n$ for every $n\in\mathbb{N}$. Then
$$
b^{qp_n}\pmod{qp_n}=b^q+kp_n
$$
Thus,
$$
\frac{b^{qp_n}\pmod{qp_n}}{qp_n}=\frac{b^q}{p_n}+\frac{k}{q},
$$
Whence, (\ref{limit}) follows.

We proceed to show that no other limit points exist. Taken any sequence $\{p_n\}_{n=1}^\infty\subset\mathbb{P}$ and assume that 
$$
\lim_{n\rightarrow\infty}\frac{b^{qp_n}\pmod{qp_n}}{qp_n}=l.
$$
Clearly, there is $x\in\mathbb{Z}^*[q(q-1)]$ such that an infinite number of elements from $\{p_n\}_{n=1}^\infty$ satisfy 
\begin{equation}
    \label{subseq}
    p_n\equiv x\pmod{q(q-1)}.
\end{equation}
Let $\{p_{n_j}\}_{j=1}^\infty$ be the sub-sequence consisting of all $p_n$ that satisfies (\ref{subseq}). Then 
$$
\lim_{j\rightarrow\infty}\frac{b^{qp_{n_j}}\pmod{qp_{n_j}}}{qp_{n_j}}=\frac{k}{q}
$$
where $k$ is the unique index such that $x\in Z_k$ (such $k$ exists due to (\ref{Disjoint})). By uniqueness of limit, we have $l=k/q$.
\end{proof}

\section{Structure of the solution sets (\ref{solSets})}
As in the proof of Lemma \ref{MainLemma}, we let $\delta=|\langle b\rangle|$ where $\langle 
 b\rangle$ is the cyclic subgroup of $\mathbb{Z}^*[q]$ generated by $b$. 
Define
\begin{equation}
    \nonumber
    \mathcal{N}=\{\alpha\in\mathbb{Z}^*[q-1]: \alpha\equiv 1\pmod{\delta}\}.
\end{equation}
\begin{theorem}
    \label{structureThm}
    The following holds.
    \begin{enumerate}
        \item For $k=0$, we have
        \begin{equation}
        \label{structureZ0}
            Z_0=\{\alpha q+(q-1)\beta\pmod{q(q-1)}:\alpha\in\mathcal{N}, \beta\in\mathbb{Z}^*[q]\}
        \end{equation}
        \item For each $\alpha\in\mathbb{Z}^*[q-1]$ and $k\in\{1,2,\ldots,q-1\}$ let $\beta_k(\alpha)$ be the unique solution in $\mathbb{Z}^*[q]$  to the equation
        \begin{equation}
            \label{betaEq}
            b-k\beta\equiv b^{\alpha\pmod{\delta}}\pmod{q}.
        \end{equation}
        Then
        \begin{equation}
            \label{structureZk}
            Z_k=\{x(\alpha):\alpha\in\mathbb{Z}^*[q-1]\setminus\mathcal{N}\}.
        \end{equation}
        where
        \begin{equation}
            \label{solDef}
            x(\alpha)=\alpha q+\beta_k(\alpha)(q-1)\pmod{q(q-1)}
        \end{equation}
    \end{enumerate}
\end{theorem}
\begin{proof}
    Take $x\in\mathbb{Z}^*[q(q-1)]$. By Lemma \ref{SimpleLemma0} \ref{SimpleLemmaConv}, we have $x=\alpha q+\beta(q-1)\pmod{q(q-1)}$.
    Assume first that $\alpha\in\mathcal{N}$, i.e., $\alpha=1+n\delta$. It follows that for any $\beta\in\mathbb{Z}^*[q]$ there holds $b^x\equiv b^q\equiv b\pmod{q}\equiv b+0\times x\pmod{q}$, so clearly, $x\in Z_0$. On the other hand, if $\alpha\notin\mathcal{N}$, then $\alpha=r+n\delta$ for $r\in\{0,2,\ldots,\delta-1\}$ and in every case $b^x\equiv b^{rq}\equiv b^{r}\not\equiv b\pmod{q}$ so for every $x=\alpha q+\beta(q-1)$ with $\alpha\notin\mathcal{N}$
    we have $x\notin Z_0$. This proves (\ref{structureZ0}).

    Take now $\alpha\in\mathbb{Z}^*[q-1]\setminus\mathcal{N}$ and $k\neq0$. Write again $\alpha=r+n\delta$ with $r\in\{0,2,\ldots,\delta-1\}$. As in the proof of Lemma \ref{MainLemma}, the equation (\ref{SetEq}) holds. Thus, for fixed $\alpha\in\mathbb{Z}^*[q-1]\setminus\mathcal{N}$, there is a unique solution $\beta_k(\alpha)\in\mathbb{Z}^*[q]$ to (\ref{betaEq}). Furthermore, $b^x\equiv b^r\pmod{q}$ so clearly $x(\alpha)$ defined by (\ref{solDef}) belongs to $Z_k$. Thus,
    $$
        \{x(\alpha):\alpha\in\mathbb{Z}^*[q-1]\setminus\mathcal{N}\}\subset Z_k.
    $$
    On the other hand, take any $x=\alpha q+\beta (q-1)\in Z_k$. (By Lemma \ref{SimpleLemma0} \ref{SimpleLemmaConv} we may assume every $x\in Z_k$ has this form.) It is clear that $\alpha\notin\mathcal{N}$, since otherwise $b^x\equiv b^q\equiv b\pmod{q}$ and then 
    $$
        b^x\equiv kx+b\pmod{q}\Leftrightarrow b\equiv kx+b\pmod{q}\Leftrightarrow kx\equiv 0\pmod{q}
    $$
    and this forces the contradiction $k=0$ since $\gcd(x,q)=1$. Thus, $\alpha\in\mathbb{Z}^*[q-1]\setminus\mathcal{N}$. A very similar argument shows that for fixed $\alpha\in\mathbb{Z}^*[q-1]\setminus\mathcal{N}$ the only possible value for $\beta$ solves (\ref{betaEq}). Since the solution to (\ref{betaEq}) is unique, this concludes the proof of (\ref{structureZk}).
\end{proof}
\begin{rem}
    \label{compRemark}
    We illustrate how to use the general method of Theorem \ref{structureThm} to determine $Z_k$ in a specific case. Take $b=2$, $q=13$ and $k=8$.
    We have $\mathbb{Z}^*[12]=\{1,5,7,11\}$ and since 2 is a primitive root modulo 13 we have $\delta=12$. Therefore $\mathcal{N}=\{1\}$ so
    $\mathbb{Z}^*[12]\setminus\mathcal{N}=\{5,7,11\}$. Consequently, for $\alpha\in\{5,7,11\}$, $\beta_8(\alpha)$ is the unique solution in $\mathbb{Z}^*[13]$ to
    $$
        8\beta\equiv 2-2^{\alpha\pmod{12}}\pmod{13}\equiv2-2^{\alpha}\pmod{13}.
    $$
    Using that $8^{-1}\pmod{13}=5$, we easily get
    $$
        \beta_8(5)=6,\quad \beta_8(7)=7,\quad \beta_8(11)=1.
    $$
    Consequently, the solutions (\ref{solDef}) are
    $$
        Z_8=\{x(5),\, x(7),\, x(11)\}=\{137, 19, 155\}.
    $$
    See also Figure \ref{table1} below.
\end{rem}

\begin{proof}[Proof of Proposition \ref{dichotomy}]
    It is instructive to prove property \ref{seconditem} first.
    Assume that $b^2\equiv 1\pmod{q}$. This means that $\delta=|\langle b\rangle|=2$, so $\mathcal{N}$ consists of the odd elements of $\mathbb{Z}^*[q-1]$. On the other hand, $q-1$ is even so $\mathbb{Z}^*[q-1]$ only contains odd elements. Consequently, $\mathbb{Z}^*[q-1]=\mathcal{N}$. Since we cannot select $\alpha\in\mathbb{Z}^*[q-1]\setminus\mathcal{N}$ the sets $Z_k~~(k=1,2,\ldots, q-1)$ are all empty. Hence, $Z_0=\mathbb{Z}^*[q(q-1)]$ and $\mathcal{S}_{b,q}'=\{0\}$.
    
    Assume now that $b^2\not\equiv 1\pmod{q}$. In this case, $\mathbb{Z}^*[q-1]\setminus\mathcal{N}\neq\emptyset$. For instance, $q-2\in\mathbb{Z}^*[q-1]$ and $q-2\equiv\delta-1\pmod{\delta}\neq1$ so $q-2\notin\mathcal{N}$). Since $\mathbb{Z}^*[q-1]\setminus\mathcal{N}$, the sets $Z_k$ given by (\ref{structureZk}) will be nonempty for all $k\in\mathbb{Z}[q]$. That is, the conclusion of Lemma \ref{MainLemma} holds; it follows that $\mathcal{S}_{b,q}'=\{k/q: k=0,1,\ldots,q-1\}$.
\end{proof}

We define the function $m_b(q)$ by 
\begin{equation}
    \label{mbDef}
    m_b(q)=|\mathcal{N}|
\end{equation}

\begin{proof}[Proof of Proposition \ref{countProp}]
    We first prove (\ref{card2}). Fix $k\in\{1,2,\ldots,q-1\}$. By Theorem \ref{structureThm}, there is a one-to-one correspondence between the elements of $\mathbb{Z}^*[q-1]\setminus\mathcal{N}$ and the elements of $Z_k$. Hence,
    $$
        |Z_k|=|\mathbb{Z}^*[q-1]\setminus\mathcal{N}|=|\mathbb{Z}^*[q-1]|-|\mathcal{N}|
    $$
    By definition, $|\mathbb{Z}^*[n]|=\varphi(n)$ (see \cite[Chapter 2]{A}) so $|\mathbb{Z}^*[q-1]|=\varphi(q-1)$ and this proves (\ref{card2}). We now prove (\ref{card1}). For this, we use (\ref{Disjoint}) to write
    \begin{equation}
        \nonumber
        \sum_{k=0}^{q-1}|Z_k|=|\mathbb{Z}^*[q(q-1)]|=\varphi(q(q-1)).
    \end{equation}
    Consequently, by (\ref{card1})
    \begin{eqnarray}
        \nonumber
        |Z_0|&=&\varphi(q(q-1))-(q-1)(\varphi(q-1)-m_b(q))\\
        \nonumber
        &=&(q-1)\varphi(q-1)-(q-1)\varphi(q-1)+(q-1)m_b(q)\\
        \nonumber
        &=&(q-1)m_b(q)
    \end{eqnarray}
    where we used $\varphi(ab)=\varphi(a)\varphi(b)$ for $a,b$ with $\gcd(a,b)=1$ together with $\varphi(q)=q-1$ for $q\in\mathbb{P}$ (see see \cite[Chapter 2]{A}). 
\end{proof}

\begin{rem}
    Take $b=2$ and $q=13$, we use Proposition \ref{countProp} to count the elements of the sets $Z_k~~(k=0,1,\ldots, 12)$. As in Remark \ref{compRemark},  $\mathbb{Z}^*[12]=\{1,5,7,11\}$ and $\mathcal{N}=\{1\}$ whence $\varphi(12)=4$ and $m_2(13)=1$. It follows that $|Z_0|=12$ and $|Z_k|=3$ for $k=1,2,\ldots, 12$. Below, we tabulate the elements of $Z_k$ in this case. The elements can be computed using the method of Theorem \ref{structureThm} as in Remark \ref{compRemark}, or by a simple computer search.
    \begin{figure}[H]
    \begin{tabular}{ c|l }
    $k\in\mathbb{Z}[13]$& $x\in\mathbb{Z}^*[156]$ belonging to $Z_k$\\
    \hline
    0&1, 25, 37, 49, 61, 73, 85, 97, 109, 121, 133, 145\\
    1&17, 83, 139 \\
    2&35, 41, 115 \\
    3&55, 71, 101\\
    4&11, 53, 103 \\
    5&7, 131, 149  \\
    6&5, 107, 151 \\
    7&23, 31, 125\\
    8&19, 137, 155\\
    9&77, 79, 119 \\
    10&29, 59, 127\\
    11&67, 89, 95\\
    12&43, 47, 113\\
    \hline
    \end{tabular}
    \caption{The solution sets $Z_0,Z_1,\ldots Z_{12}$ when $b=2,q=13$}
    \label{table1}
    \end{figure}
\end{rem}

\subsection*{Acknkowledgements}
The author is grateful to Professor A. Dubickas (Vilnius) for pointing out the references \cite{CKLRS}-\cite{D2}.


\begin{thebibliography}{0}
\bibitem{A} T. M. Apostol, {\it Introduction to Analytic Number Theory}, Springer, New York, NY, 1976

\bibitem{CKLRS} J. Cilleruelo, A. Kumchev, F. Luca,
J. Ru\'{e} and I. E. Shparlinski, "On the fractional parts of $a^n/n$", Bull. London Math. Soc. {\bf 45} (2013), 249-256

\bibitem{D1} A. Dubickas, "Density of some sequences modulo 1", Colloq. Math. {\bf 128} (2012), 237–244

\bibitem{D2} A. Dubickas, "Density of some special sequences modulo 1", Mathematics {\bf 11} (2023), 1727
\end{thebibliography}
\end{document}